\documentclass[12pt]{amsart}
\usepackage{amssymb,amsmath,amsthm,bm,stmaryrd }
\usepackage{hhline}
\usepackage{amsrefs}
\usepackage[pdftex]{graphicx}

\theoremstyle{plain}
\newtheorem{Theorem}{Theorem}[section]
\newtheorem*{mainthm*}{Main Theorem}

\newtheorem{Lemma}[Theorem]{Lemma}

\theoremstyle{definition} 
\newtheorem{Definition}[Theorem]{Definition}

\newtheorem{Example}[Theorem]{Example}

\def\Z{{\mathbb{Z}}}
 
\def\Q{{\mathbb{Q}}} 
\def\ch{{\mathrm{char}}} 

\title{On finite generating sets of infinitely generated ideals}
\author{Takafumi Shibuta\\Kyushu Sangyo University}
\address{Department of Information Science, Kyushu Sangyo University}
\date{}
\thanks{This work has been supported by JSPS Grant-in-Aid for Scientific Research (C) (No. 22K03334).}
\email{tshibuta@ip.kyusan-u.ac.jp}
\begin{document}
\maketitle
\begin{abstract}
This paper presents a novel approach to constructing 
finite generating sets for infinitely generated ideals. 
By integrating algebraic and computational techniques, 
we provide a method to identify finite generators, 
demonstrated through illustrative examples.
\end{abstract}
\maketitle
\section{Introduction}
The Hilbert basis theorem \cite{Hilbert} asserts that any polynomial ideal is finitely generated. Nonetheless, determining a finite generating set for an ideal derived from infinitely many polynomials poses significant theoretical and computational challenges. 
In their work \cite{BE}, Buchberger and El\'ias demonstrate that the ideal generated by the Fermat polynomials $f_n=x^n+y^n-z^n$ for $n \ge 3$ is, in fact, generated by $f_3$, $f_4$, and $f_5$. 
In this paper, we  discuss more general situations. 
For example, take the ideal of the polynomial ring $\Q[x,y,z,w]$ generated by infinitely many polynomials 
\begin{align*}
f_n=(2n^6+3)xy^{n^3+2}+(n^3+1)z^{n^3+2}+w^{n^3+2},
n \in \Z_{\ge 0}. 
\end{align*} 
Later, we will show that this ideal is finitely generated by  $f_0,f_1,f_2,$ and $f_3$ (Example \ref{Ex cubic2}). 

We first consider ideals generated by polynomials 
$f_{n_1,\dots,n_p}$ where the parameters $n_1,\dots,n_p$ appear in the coefficients and exponents in polynomial forms, 
and range over the non-negative integers. 
However, this general form of the problem is not solvable even in the case of a polynomial ring with a single variable: 
Let $\varphi(n_1,\dots,n_p)\in \Z[n_1,\dots,n_p]$ be a polynomial, 
and consider the following ideal of $\Q[x]$ 
\begin{align*}I=\langle x^{\varphi(n_1,\dots,n_p)^2}\mid n_1,\dots,n_p \in \Z_{\ge 0} \rangle\subset \Q[x]. \end{align*}
Then $I$ is generated by $x^M$ where $M=\min \{\varphi(n_1,\dots,n_p)^2\mid n_1,\dots,n_p \in \Z_{\ge 0} \}$. 
In particular, $I=\Q[x]$ if and only if the Diophantine equation $\varphi(n_1,\dots,n_p)=0$ has a non-negative integer solution. 
Therefore, determining a generating set of $I$ depends on the solvability of the Diophantine equation $\varphi(n_1, \dots, n_p) = 0$. According to the MRDP (Matiyasevich-Robinson-Davis-Putnam) theorem \cite{Matiyasevich}, no algorithm exists to decide this in general.

Therefore, constraints must be imposed on the exponents to facilitate the derivation of finite generating sets.
Beginning with the simplest case where the exponents are linear, we provide explicit methods to identify finite generating sets. 

\begin{mainthm*}[Theorem \ref{main}]
Fix positive integers $p,r$, and $N$. 
Let $R$ be a commutative ring, and $x_1,\dots,x_r \in R$. 
Let 
\begin{align*}
c_k(n_1,\dots,n_p)\in R[n_1,\dots,n_p], 1\le k \le N,
\end{align*}
be polynomials in the parameters $n_1,\dots,n_p$, 
and let $\alpha^{(k)}_{ij},\beta^{(k)}_{i}\in \Z_{\ge 0}$ be non-negative integers for $1\le i \le r, 1\le j \le p ,1\le k \le N$. 
Let 
\begin{align*}
f_{n_1,\dots,n_p}=\sum_{k=0}^N c_k(n_1,\dots,n_p) \prod_{i=1}^r x_i^{\sum_{j=1}^p \alpha^{(k)}_{ij} n_j+\beta^{(k)}_{i}},
\end{align*}
and consider the ideal $I=\langle f_{n_1,\dots,n_p} \mid {n_1,\dots,n_p}\in \Z_{\ge 0} \rangle \subset R$. 
We set 
\begin{align*}
\ell_i=\sum_{k=1}^N(\deg_{n_i} c_k+1)-1.
\end{align*}
Then $I$ is generated by $\{ f_{n_1,\dots,n_p} \mid \forall i,  0\le n_i \le \ell_i\}$. 
\end{mainthm*}
In the last section, we will show how to apply the above theorem 
along with the Gr\"obner basis theory in the case of higher degree exponents. 
\section{Main theorem}
Let $R$ be a commutative ring (not necessarily Noetherian). 
For $\bm{x}=(x_1,\dots, x_r) \in R^r$ and an exponent vector 
$\bm{\alpha}=(\alpha_1,\dots,\alpha_r)\in \Z_{\ge 0}^r$, 
we use the multi-index notation $\bm{x^\alpha}=x_1^{\alpha_1}\cdots x_r^{\alpha_r}$, and $\bm{\alpha!}=(\alpha_1 !)\times \dots \times (\alpha_r !)$. For example, 
let $R=\Q[x,y]$, $\bm{z}=(2,x,x+y)\in R^3$, 
$\bm{\alpha}=(n+2,n,2n+1)$, and $\bm{\beta}=(n,0,0)$, 
then $\bm{z}^{\bm{\alpha}}+\bm{z}^{\bm{\beta}}=2^{n+2}x^n(x+y)^{2n+1}+2^n$. 

We define the product ordering $\le$ on $\Z_{\ge 0}^r$ as 
\begin{align*}
\bm{\alpha}=(\alpha_1,\dots,\alpha_r) \le \bm{\beta}=(\beta_1,\dots,\beta_r)
\Leftrightarrow 1\le \forall i \le r,\alpha_i\le \beta_i. 
\end{align*}
For example, $(2,7,4)\le (3,7,6)$, and $(1,3,1)$ and $(4,1,4)$ are not comparable. 

$R\llbracket \bm{t}\rrbracket=R\llbracket t_1,\dots,t_r\rrbracket$ denotes the formal power series ring over $R$. 
For a infinite collection $\{ f_{\bm{n}}\in R \mid \bm{n}\in \Z_{\ge 0}^r\}$, we consider its generating function $\sum_{\bm{n}\in \Z_{\ge 0}^r} f_{\bm{n}} \bm{t}^{\bm{n}}$.
Note that $\psi(\bm{t}) \in R\llbracket \bm{t} \rrbracket$ is invertible if and only if $\psi(\bm{0})$ is invertible in $R$. 

The next lemma is the key to prove the main theorem..
\begin{Lemma}\label{key}
Fix a vector $\bm{\beta}_0=(\beta_1,\dots,\beta_r) \in \Z_{\ge 0}^r$. 
Let $f_{\bm{n}}\in R$ for $\bm{n}\in \Z_{\ge 0}^r$, 
and $g_{\bm{\alpha}}\in R$ for $\bm{\alpha}\in \Z_{\ge 0}^r, \bm{\alpha} \le \bm{\beta}_0$. 
If there exists an invertible element $\psi(\bm{t}) \in R\llbracket \bm{t} \rrbracket$ such that
\begin{align*}
\sum_{\bm{n}\in \Z_{\ge 0}^r} f_{\bm{n}} \bm{t}^{\bm{n}}
=\psi(\bm{t}) \sum_{\bm{\alpha}\in \Z_{\ge 0}^r, \bm{\alpha} \le \bm{\beta}_0} g_{\bm{\alpha}} \bm{t}^{\bm{\alpha}}, 
\end{align*}
then the ideal 
$\langle f_{\bm{n}}\in R\mid \bm{n}\in \Z_{\ge 0}^r\rangle\subset R$ 
is generated by the finite set 
\begin{align*}\{ f_{\bm{n}} \mid \bm{n}\in \Z_{\ge 0}^r, \bm{n} \le \bm{\beta}_0\}.
\end{align*} 
\end{Lemma}
\begin{proof}
By expanding the right-hand side and comparing the coefficients on both sides of the equation, we conclude that 
$f_{\bm{n}}\in \langle g_{\bm{\alpha}} \mid \bm{\alpha}\in \Z_{\ge 0}^r, \bm{\alpha} \le \bm{\beta}_0\rangle$
for all $\bm{n}\in \Z_{\ge 0}^r$. 
Thus 
$\langle f_{\bm{n}}\mid \bm{n}\in \Z_{\ge 0}^r\rangle
\subset 
\langle g_{\bm{\alpha}} \mid \bm{\alpha}\in \Z_{\ge 0}^r, \bm{\alpha} \le \bm{\beta}_0\rangle$. 

Since $\psi$ has the inverse element $\psi(\bm{t})^{-1} \in R\llbracket \bm{t}\rrbracket$, we have 
\begin{align*}
\psi(\bm{t})^{-1} \sum_{\bm{n}\in \Z_{\ge 0}^r} f_{\bm{n}} \bm{t}^{\bm{n}}
=\sum_{\bm{\alpha}\in \Z_{\ge 0}^r, \bm{\alpha} \le \bm{\beta}_0} g_{\bm{\alpha}} \bm{t}^{\bm{\alpha}}. 
\end{align*}
By expanding the left-hand side and comparing the coefficients of $\bm{t}^{\bm{\alpha}}$ with $\bm{\alpha}\le \bm{\beta}_0$ on both sides of the equation, we conclude that 
$g_{\bm{\alpha}}\in \langle f_{\bm{n}} \mid \bm{n}\in \Z_{\ge 0}^r, \bm{n} \le \bm{\beta}_0\rangle$
for all $\bm{\alpha}\le \bm{\beta}_0$. 
Thus, we have 
$\langle g_{\bm{\alpha}} \mid \bm{\alpha}\in \Z_{\ge 0}^r, \bm{\alpha} \le \bm{\beta}_0\rangle
\subset 
\langle f_{\bm{n}}\mid \bm{n}\in \Z_{\ge 0}^r, \bm{n} \le \bm{\beta}_0\rangle$. 

Summing up the above, 
\begin{align*}\langle f_{\bm{n}}\mid \bm{n}\in \Z_{\ge 0}^r\rangle
\subset 
\langle g_{\bm{\alpha}} \mid \bm{\alpha}\in \Z_{\ge 0}^r, \bm{\alpha} \le \bm{\beta}_0\rangle
\subset 
\langle f_{\bm{n}}\mid \bm{n}\in \Z_{\ge 0}^r, \bm{n} \le \bm{\beta}_0\rangle.
\end{align*}
This proves the assertion. 
\end{proof}
\begin{Example}
Let $f_n=xy^{2n+1}+zw^{2n+1} \in \Q[x,y,z,w]$ 
for $n\in \Z_{\ge 0}$. 
Then 
\begin{align*}
\sum_{n=0}^\infty f_n t^n
&=xy\sum_{n=0}^\infty (y^2t)^n+zw\sum_{n=0}^\infty (w^2t)^n
=\frac{xy}{1-y^2t}+\frac{zw}{1-w^2t}\\
&=\frac{(xy+zw)-(xw+yz)ywt}{(1-y^2t)(1-w^2t)}. 
\end{align*}
Since $\frac{1}{(1-y^2t)(1-w^2t)}$ is invertible, 
$\langle f_n \mid n \in \Z_{\ge 0} \rangle$ is generated by $f_0, f_1$ by Lemma \ref{key}. 
This ideal appears as the defining ideal for the set of true parameters in a specific probabilistic model \cite{Watanabe}.
\end{Example}
\begin{Definition}
For a variable $w$ and $k\in \Z_{\ge 0}$, we define 
\begin{align*}
[w]_k= w(w - 1) \dots (w - k + 1).
\end{align*}
\end{Definition} 
It is well-known that the generating function of $[n]_k,~n\in \Z_{\ge 0},$ is a rational function. 
\begin{Theorem}\label{generating function}
Let $k \in \Z_{\ge 0}$. 
For any commutative ring $R$,  
\begin{align*}
\sum_{n=0}^\infty [n]_k t^n =\dfrac{k! t^k}{(1-t)^{k+1}}
\end{align*}
holds in $R\llbracket t\rrbracket$. 
\end{Theorem}
\begin{proof}
It suffices to prove the case where $R=\Z$. 
By applying $t^k\dfrac{d^k}{dt^k}$ on the both sides of 
\begin{align*}
\sum_{n=0}^\infty t^n =\dfrac{1}{1-t}, 
\end{align*}
we conclude the assertion. 
\end{proof}
The characteristic of $R$, denoted by $\ch R$, is the non-negative generator of the kernel of the unique ring homomorphism $\Z\to R$. 
When $\ch R > 0$ and $k \ge \ch R$, it holds that $k!=0$ and $[n]_k =0$ for all $n \in \Z_{\ge 0}$ in $R$. 
Thus, in this case, $\sum_{n=0}^\infty [n]_k t^n =\frac{k! t^k}{(1-t)^{k+1}}=0$. 

We are now ready to prove the main theorem of this paper. 
\begin{Theorem}\label{main}
Fix positive integers $p,r$, and $N$. 
Let $R$ be a commutative ring, and $\bm{x}=(x_1,\dots,x_r )\in R^r$. 
Let $c_k(\bm{n})\in R[\bm{n}]= R[n_1,\dots,n_p]$  be polynomials in the parameters $\bm{n}=(n_1,\dots,n_p)$, 
and $\bm{A}^{(k)}=(\alpha^{(k)}_{ij})\in \Z_{\ge 0}^{r\times p} ,\bm{\beta}^{(k)}=(\beta^{(k)}_{i})\in \Z^r_{\ge 0}$ for $1\le k \le N$. 
Let 
\begin{align*}
f_{\bm{n}}=\sum_{k=0}^N c_k(\bm{n}) \bm{x}^{\bm{A}^{(k)}\bm{n}+\bm{\beta}^{(k)}},~~
\ell_j=\sum_{k=1}^N(\deg_{n_j} c_k+1)-1,
\end{align*}
and $\bm{\ell}=(\ell_1,\dots,\ell_p)$ where $\deg_{n_j} c_k$ denotes the degree of $c_k$ in $n_j$. 

Then the ideal $\langle f_{\bm{n}} \mid \bm{n}\in \Z^r_{\ge 0} \rangle \subset R$ is generated by 
$\{ f_{\bm{n}} \mid \bm{n}\le \bm{\ell}\}$. 
\end{Theorem}
\begin{proof}
Let $\bm{\alpha}_j^{(k)}=(\alpha^{(k)}_{1j}, \dots, \alpha^{(k)}_{rj})$. 
Then 
\begin{align*}
\bm{x}^{\bm{A}^{(k)}\bm{n}+\bm{\beta}^{(k)}}
=\bm{x}^{\bm{\beta}^{(k)}}
\prod_{j=1}^p \left(\bm{x}^{\bm{\alpha}_j^{(k)}}\right)^{n_j}.
\end{align*}
For each $k$, the coefficient polynomial $c_k(\bm{n})$ can be written as 
\begin{align*}
c_k(\bm{n}) = c_k(n_1,\dots,n_p)
=\sum_{\substack{\bm{m}=({m}_1,\dots, {m}_p)\in \Z_{\ge 0}^p \\
 \forall j, 0 \leq {m}_j \leq \deg_{n_j}c_k}}
 \gamma_{\bm{m}}
 \prod_{j=1}^p [n_j]_{{m}_j}
\end{align*}
for some $\gamma_{\bm{m}}\in R$. 
By Lemma \ref{key}
\begin{align*}
\sum_{\bm{n}\in \Z_{\ge 0}^p}
\prod_{j=1}^p [n_j]_{{m}_j} \bm{x}^{\bm{A}^{(k)}\bm{n}+\bm{\beta}^{(k)}}\bm{t}^{\bm{n}}
&=\bm{x}^{\bm{\beta}^{(k)}}
\prod_{j=1}^p 
\left\{
\sum_{n_j=0}^\infty
[n_j]_{{m}_j}\left(\bm{x}^{\bm{\alpha}_j^{(k)}}t_j\right)^{n_j}
\right\}\\
&=\frac{\bm{m!}\bm{x}^{\bm{\beta}^{(k)}}\bm{t}^{\bm{m}}}
{\prod_{j=1}^p \left(1-\bm{x}^{\bm{\alpha}_j^{(k)}}t_j\right)^{{m}_j+1}}
\end{align*}
Hence  
\begin{align*}
\sum_{\bm{n}\in \Z_{\ge 0}^p} f_{\bm{n}}\bm{t}^{\bm{n}}
&=
\sum_{k=1}^N \sum_{\bm{n}\in \Z_{\ge 0}^p} c_k(\bm{n}) \bm{x}^{\bm{A}^{(k)}\bm{n}+\bm{\beta}^{(k)}} \bm{t}^{\bm{n}}\\
&=\sum_{k=1}^N \sum_{\substack{\bm{m}=({m}_1,\dots, {m}_p)\in \Z_{\ge 0}^p \\
 \forall j, 0 \leq {m}_j \leq \deg_{n_j}c_k}}
 \frac{\bm{m!}\gamma_{\bm{m}}\bm{x}^{\bm{\beta}^{(k)}}\bm{t}^{\bm{m}}}
{\prod_{j=1}^p \left(1-\bm{x}^{\bm{\alpha}_j^{(k)}}t_j\right)^{{m}_j+1}}.
\end{align*}
By putting all terms over a common denominator, we have  
\begin{align*}
\sum_{\bm{n}\in \Z_{\ge 0}^p} f_{\bm{n}}\bm{t}^{\bm{n}}=\frac{g(\bm{t})}
{\prod_{k=1}^N\prod_{j=1}^p 
\left(1-\bm{x}^{\bm{\alpha}_j^{(k)}}t_j\right)^{\deg_{n_j}c_k+1}},
\end{align*}
where the numerator $g(\bm{t})$ satisfies 
\begin{align*}
\deg_{t_j} g(\bm{t})< \deg_{t_j} \prod_{k=1}^N\prod_{j=1}^p 
\left(1-\bm{x}^{\bm{\alpha}_j^{(k)}}t_j\right)^{\deg_{n_j}c_k+1}
=\sum_{j=1}^p \left(\deg_{n_j}c_k+1 \right).
\end{align*}
This shows that any term $\bm{t}^{\bm{m}}$ appearing in $g(\bm{t})$ with non-zero coefficient satisfies $\bm{m}\le \bm{\ell}$. 
Since the denominator $ \prod_{k=1}^N\prod_{j=1}^p 
\left(1-\bm{x}^{\bm{\alpha}_j^{(k)}}t_j\right)^{\deg_{n_j}c_k+1}$ 
is invertible in $R\llbracket \bm{t}\rrbracket$, 
the ideal $\langle f_{\bm{n}} \mid \bm{n}\in \Z^r_{\ge 0} \rangle \subset R$ is generated by 
$\{ f_{\bm{n}} \mid \bm{n}\le \bm{\ell}\}$ by Lemma \ref{key}. 
\end{proof}
Note that when $\ch R >0$, we can reduce the situation to the case where $\deg_{n_j} c_k < \ch R$ for all $j$ and $k$.
\begin{Example}
Let $R=\Q[x,y,z]$ and
\begin{align*}
f_n&=x^{2n+1}+y^{2n+1}+z^{n+2}+2^{n+2}x^2 
\end{align*}
for $n\in \Z_{\ge 0}$. 
Since degree of the coefficients $1,1,1,1$ are all $0$, 
we have $\ell=4-1=3$. 
By Theorem \ref{main}, 
$\langle f_n \mid n\in \Z_{\ge 0}\rangle$ 
is generated by $f_0,\dots,f_3$. 
\end{Example}
\begin{Example}
Let $R=\Q[x_1,x_2,x_3,y_1,y_2,y_3]$ and 
\begin{align*}
f_n=(n+1)x_1y_1^{2n+1}+x_2y_2^{n+3}+(n^2+1)x_3y_3^{n+2}
\end{align*}
for $n\in \Z_{\ge 0}$. 
Since the degree of the coefficients $n+1,1,n^2+1$ are 
$1,0,2$ respectively, we have $\ell=(1+1)+(0+1)+(2+1)-1=5$. 
Thus $\langle f_n \mid n\in \Z_{\ge 0}\rangle$ 
is generated by $f_0,\dots,f_5$  by Theorem \ref{main}. 
\end{Example}
\begin{Example}
Let $R=\Q[x,y,z]$ and 
\begin{align*}
f_{n,m}=(n^2-m)x^{n+m+2}+(nm^4+m^2)x^{2n+m+1}z^{m+n+3}+(n^5+n^2)y^{4n+m+2} \in R.
\end{align*}
Then the degree of the coefficients 
$n^2-m, nm^4+m^2, n^5+n^2$
in $n$ (resp. $m$) are 
$2,1,5$ (resp. $1,4,0$). 
Thus 
\begin{align*}
\ell_1&=(2+1)+(1+1)+(5+1)-1=10,\\
\ell_2&=(1+1)+(4+1)+(0+1)-1=7,
\end{align*}
and $\langle f_{n,m}\mid n,m \in \Z_{\ge 0} \rangle$ is generated by $\{f_{n,m}\mid 0\le n \le 10, 0\le m\le 7\}$ by Theorem \ref{main}. 
\end{Example}
We remark that a generating set obtained by Theorem \ref{main} is not always minimal. 
\begin{Example}
Let $f_n=x^{2n+1}+(n^{100}+1)y^{n+1}$ for $n\in \Z_{\ge 0}$. 
The generating set of $\langle f_n \mid n\in \Z_{\ge 0}\rangle$ obtained by Theorem \ref{main} is $f_0, \dots,f_{101}$. 
On the other hand,  $\langle f_0,f_1,f_2 \rangle=\langle x+y,y^2\rangle$, and furthermore, it can be easily shown that $f_n \in \langle x+y,y^2\rangle$ for all $n\in \Z_{\ge 0}$. 
\end{Example}
\section{The case of higher degree exponents}
In this section, we consider the case where the exponents are polynomials in parameters of degree $\ge 2$. 
Since $\sum_{n=0}^\infty t^{n^d}$ is not a rational function for $d\ge 2$, 
we cannot apply the method used in the proof of Theorem \ref{main} in this cases. 

We will propose an approach to higher degree exponents case using Gr\"obner basis (\cite{Buchberger}\cite{CLO}) and Theorem \ref{main}. 
By using Gr\"obner bases and the division algorithm, the ideal membership problem for finitely generated ideals can be solved (see \cite{CLO} Section 2.8). 
The idea is as follows: 
Let $\bm{n}=(n_1,\dots,n_{p})$ and $\bm{m}=(m_1,\dots,m_{q})$ be parameter vectors, 
and let $\varphi_j(\bm{m})\in \Z[\bm{m}]$ for $1\le j \le q$. 
Let 
\begin{gather*}
g_{\bm{m}}=\sum_{k=0}^N c_k(m_1,\dots,m_{q}) \prod_{i=1}^r x_i^{\sum_{j=1}^{q} \alpha^{(k)}_{ij} m_j+\beta^{(k)}_{i}}\in R[x_1,\dots,x_r], \\
f_{\bm{n}}=g_{(\varphi_1(\bm{n}),\dots,\varphi_q(\bm{n})) }=\sum_{k=0}^N c_k(\varphi_1(\bm{n}),\dots,\varphi_q(\bm{n})) \prod_{i=1}^r x_i^{\sum_{j=1}^q \alpha^{(k)}_{ij} \varphi_j(\bm{n})+\beta^{(k)}_{i}}, 
\end{gather*}
and $I=\langle f_{\bm{n}} \mid \bm{n}\in \Z_{\ge 0}^p \rangle$. 
We consider a finite generating set of $I$. 
Fix a sutable $\bm{k}_0 \in \Z_{\ge 0}^r$ and let $J=\langle f_{\bm{n}} \mid \bm{n}\le \bm{k}_0 \rangle$.
Then take a sufficient large $\bm{k}_1  \in \Z_{\ge 0}^r$ with $\bm{k}_0 < \bm{k}_1$, 
and compute the list 
\begin{align*}
L=\{\bm{m} \in \Z_{\ge 0}^q \mid \bm{m}\le \bm{k}_1,~ g_{\bm{m}} \in J\} 
\end{align*}
by using Gr\"obner bases. 
It sometimes happens that the list $L$ contain $\bm{m}$'s which are not written as $\bm{m}=(\varphi_1(\bm{n}),\dots,\varphi_q(\bm{n}))$ for any $\bm{n}\in \Z_{\ge 0}^p$. 
If we are fortunate, we can conclude that $ f_{\bm{n}}\in J$ for all $\bm{n}$ using Theorem \ref{main}, which proves $I=J$. 
\begin{Example}
Let $R=\Q[x,y,z,w]$, and 
\begin{align*}
g_{n} &=xy^{n+1}+z^{n+1}+w^{n+1},\\
f_{n}=g_{n^3} &=xy^{n^3+1}+z^{n^3+1}+w^{n^3+1}.
\end{align*}
Let $I_1=\langle f_{n}\mid n\in\Z_{\ge 0} \rangle\subset R$ and 
$J_1=\langle f_0,f_1,f_2,f_3\rangle=\langle g_0,g_1,g_8,g_{27}\rangle$.  
Using Gr\"obner basis, we can determine that $g_{28}, g_{29} \in J_1$ 
even though $28, 29$ are not cubic numbers. 

On the other hand, by Theorem \ref{main}, $\langle g_{n+27} \mid n\in\Z_{\ge 0} \rangle$ is generated by $g_{27},g_{28}, g_{29}$. 
Hence $g_n \in J_1$ for all $n\ge 27$. In particular, $f_n=g_{n^3}\in J_1$ for all $n\ge 3$. 
In conclusion, we have $I_1=J_1$.  
\end{Example}
\begin{Example}\label{Ex cubic2}
Let $R=\Q[x,y,z,w]$, and 
\begin{align*}
g_n &=(2n^2+3)xy^{n+2}+(n+1)z^{n+2}+w^{n+2},\\
f_n=g_{n^3} &=(2n^6+3)xy^{n^3+2}+(n^3+1)z^{n^3+2}+w^{n^3+2}.
\end{align*}
Let $I_2=\langle f_{n}\mid n\in\Z_{\ge 0} \rangle\subset R$ and 
$J_2=\langle f_0,f_1,f_2,f_3\rangle=\langle g_0,g_1,g_8,g_{27}\rangle$.  
By using Gr\"obner basis, we can determine that $g_{n} \in J_2$ 
for $n=27,28,29,30,31,32$. 
By Theorem \ref{main}, $\langle g_{n+27} \mid n\in\Z_{\ge 0} \rangle$ is generated by $ g_{n+27}$ for $0\le n \le 5$. 
Hence $g_n \in J_2$ for all $n\ge 27$. In particular, $f_n=g_{n^3}\in J_2$ for all $n\ge 3$. 
In conclusion, we have $I_2=J_2$. 
\end{Example}
\begin{Example}
Let $R=\Q[x,y,z,w]$, and 
\begin{align*}
f_{n} &=xy^{n+1}+z^{n+1}+w^{n+1},\\
g_{n}=f_{n^2} &=xy^{n^2+1}+z^{n^2+1}+w^{n^2+1}.
\end{align*}
$I_3=\langle f_{n}\mid n\in\Z_{\ge 0} \rangle\subset R$ and 
$J_3=\langle f_0,f_1,f_2,f_3\rangle=\langle g_0,g_1,g_2,g_9\rangle$.  
Figure \ref{fig:ex_sq} shows the values of $n$ that satisfy $g_{n} \in J_3$, 
computed using Gr\"obenr basis. 
By Theorem \ref{main}, $g_9,g_{12},g_{15} \in J_3$ and $g_{10},g_{13},g_{16} \in J_3$ implies $g_{3n+9}, g_{3n+10}\in J_3$ for all $n\in \Z_{\ge 0}$. 
Hence $g_n \in J_3$ for all $n\ge 9$ with $n\not\equiv 2 \bmod 3$. 
Since $n^2 \not\equiv 2 \bmod 3$ for all $n$, we conclude that $f_n \in J_3$ for all $n\ge 3$. In conclusion, we have $I_3=J_3$. 
\begin{figure}[hbpt]
    \centering
\includegraphics[scale=0.6]{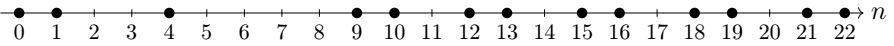}
    \caption{Plot of $n$ satisfying $g_{n}\in J_3$.}
    \label{fig:ex_sq}
\end{figure}
\end{Example}
\begin{Example}
Let $R=\Q[x,y,z,u,v]$, and 
\begin{align*}
g_{n_1,n_2} &=x^{2n_1+3}+y^{2n_2+3}+u^{n_1+2}v^{n_2+2}\\
f_{n,m}=f_{n^2+m^2,m^3} &=x^{2n^2+2m^2+3}+y^{2m^3+3}+u^{n^2+m^2+2}v^{m^3+2}.
\end{align*}
We will find a finite generating set of the ideal $I_4=\langle f_{n,m}\mid n,m\in\Z_{\ge 0} \rangle\subset R$. 
Let $J_4=\langle f_{n,m} \mid 0\le n \le 3, 0\le m \le 2\rangle$. 
Figure \ref{fig:ex_nm0} shows the pairs $(n_1,n_2)$ that satisfy $g_{n_1,n_2} \in J_4$. 
Since $g_{n_1+1,n_2}\in J_4$ for $0\le n_1, n_2\le2$, we have 
$g_{n_1+1,n_2}\in J_4$ for all $n_1,n_2\in \Z_{\ge 0}$ by Theorem \ref{main}. 
Since $n^2+m^2\ge 1$ for $(n,m)\neq (0,0)$, 
it follows that $f_{n,m}\in J_4$ for all $(n,m)\in \Z_{\ge 0}$. 
In conclusion, we have $I_4=J_4$. 
\begin{figure}[hbpt]
    \centering
\includegraphics[scale=0.6]{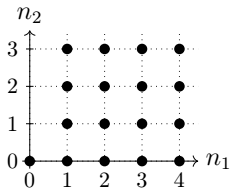}
    \caption{Plot of $(n_1,n_2)$ satisfying $g_{n_1,n_2}\in J_4$.}
    \label{fig:ex_nm0}
\end{figure}
\end{Example}
\begin{Example}
Let $R=\Q[x,y,z,u,v]$, and 
\begin{align*}
g_{n_1,n_2} &=x^{2n_1+3}+y^{2n_2+3}+u^{n_1+2}v^{n_2+2}\\
f_{n,m}=g_{n^2+m^2,n+m^3} &=x^{2n^2+2m^2+3}+y^{2m^3+n+3}+u^{n^2+m^2+2}v^{m^3+n+2}.
\end{align*}
We will find a finite generating set of the ideal $I_5=\langle g_{n,m}\mid n,m\in\Z_{\ge 0} \rangle\subset R$. 
Let 
\begin{align*}
J_5&=\langle f_{n,m} \mid 0\le n\le 3, 0\le m \le 2\rangle\\
&=\langle g_{0,0},g_{1,1},g_{2,2},g_{4,2},g_{4,8},g_{5,3},g_{5,9},
g_{8,10},g_{9,3},g_{10,4},g_{13,11}\rangle. 
\end{align*}
Figure \ref{fig:ex_nm0} shows the pairs $(n_1,n_2)$ that satisfy $g_{n_1,n_2} \in J_5$. 

By Theorem \ref{main}, it follows that 
$g_{2n_1+3,2n_2+3}, g_{2n_1+4,2n_2+4} \in J_5$
for all $n_1,n_2\in \Z_{\ge 0}$. 
Thus $g_{n_1,n_2} \in J_5$ for all $n_1,n_2\in \Z_{\ge 0}$ 
satisfying $n_1\equiv n_2 \bmod 2$ and $(n_1,n_2)\ge (3,3)$. 

Since $n^2 + m^2 \equiv m^3 + n \bmod 2$ and $(n^2 + m^2, m^3 + n) \ge (3, 3)$ for all $(n,m)\in \Z_{\ge 0}^2$ with the exceptions of $(n, m) = (0, 0), (0, 1), (1, 0), (1, 1), (2, 0)$, it follows that $f_{n,m} \in J_5$ for all $n, m \in \mathbb{Z}_{\ge 0}$. 
In conclusion, we have $I_5=J_5$. 
\begin{figure}[hbpt]
    \centering
\includegraphics[scale=0.6]{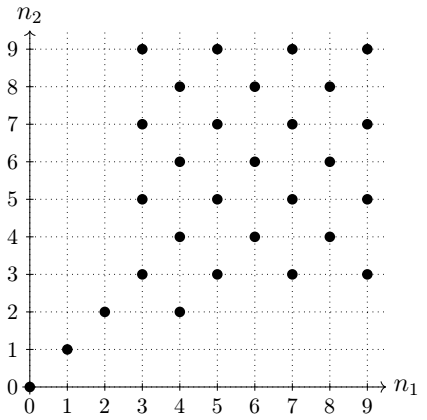}
    \caption{Plot of $(n_1,n_2)$ satisfying $g_{n_1,n_2}\in J_5$.}
    \label{fig:ex_nm}
\end{figure}
\end{Example}
\begin{Example}
\begin{align*}
g_{n_1,n_2} &=x^{n_1+1}y+z^{n_2+1}+w^{n_1+1},\\
f_{n}=g_{n^2,n} &=x^{n^2+1}y+z^{n+1}+w^{n^2+1}.
\end{align*}
Let $I_6=\langle f_{n}\mid n\in\Z_{\ge 0} \rangle\subset R$ and 
$J_6=\langle f_0,f_1,f_2,f_3,f_4\rangle$. 
Figure \ref{fig:ex_pa} shows the pairs $(n_1,n_2)$ that satisfy $g_{n_1,n_2} \in J_6$ and the curve $(t^2,t),~t\ge 0$. 
Since $g_{6m_1+7,2m_2+3}\in J_6$ for $(m_1,m_2) \le (2,2)$, we have $g_{6m_1+7,2m_2+3}\in J_6$ for all $(m_1,m_2) \in \Z_{\ge 0}^2$ by Theorem \ref{main}. 
Similarly, we have 
\begin{align*}
g_{6m_1+9,2m_2+3},~g_{6m_1+10,2m_2+4},~g_{6m_1+12,2m_2+4}\in J_6
\end{align*}
for all $(m_1,m_2) \in \Z_{\ge 0}^2$. 
This indicates that $g_{n_1,n_2} \in J_6$ for all $(n_1,n_2)\in\Z_{\ge 0}^2$ satisfying the following conditions 
\begin{align*}
(n_1,n_2)\ge(7,3),~~n_1\equiv n_2\bmod 2,~~n_1\not\equiv 2 \bmod 3.
\end{align*}
Since $n^2 \not\equiv 2 \bmod 3$, we have $f_n \in J_6$ for all $n\ge 3$. 
In conclusion, we have $I_6=J_6$. 
\begin{figure}[htbp]
    \centering
\includegraphics[scale=0.65]{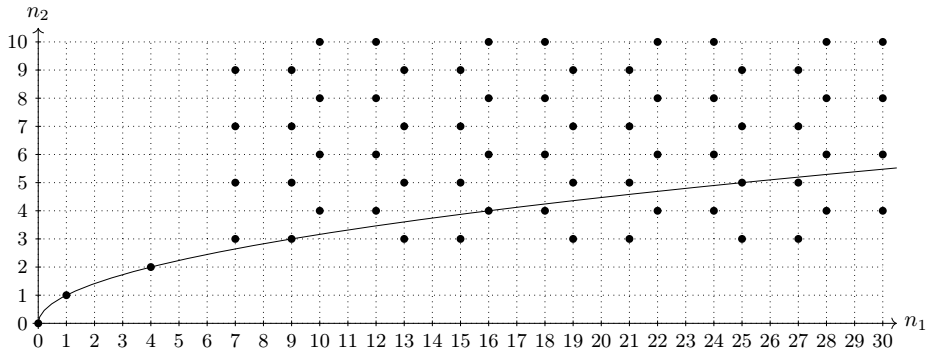}
    \caption{Plot of $(n_1,n_2)$ satisfying $g_{n_1,n_2}\in J_6$, and the curve $(t^2,t), t\ge 0$.}
    \label{fig:ex_pa}
\end{figure}
\end{Example}

\end{document}